\documentclass[11pt,reqno]{amsart}

\usepackage{amssymb}
\usepackage{enumerate}
\usepackage{color}
\newtheorem{theorem}{Theorem}[section]
\newtheorem{question}{Question}[section]

\newtheorem{prop}[theorem]{Proposition}
\newtheorem{lemma}[theorem]{Lemma}

\theoremstyle{remark}
\newtheorem{remark}{Remark}[section]

\numberwithin{equation}{section}
\DeclareMathOperator\supp{\rm supp}

\def\N{\mathbb{N}}
\def\R{\mathbb{R}}
\def\Z{\mathbb{Z}}

 \usepackage{ulem}
\begin{document}

\title[The continuity of Beurling density and Beurling dimension]{The continuity of Beurling density and Beurling dimension of spectra of a class of self-affine spectral measures}

\author{Zi-Yun Chen}
\address[Zi-Yun Chen]{School of Mathematics and Information Science, Guangzhou University, Guangzhou, 510006, P.~R.~China}
\email{2112415042@e.gzhu.edu.cn}

\author{Zhi-Yi Wu$^\ast$}
\address[Zhi-Yi Wu]{School of Mathematics and Information Science, Guangzhou University, Guangzhou, 510006, P.~R.~China}
\email{zhiyi\_wu2021@163.com}

\author{Min-Min Zhang}
\address[Min-Min Zhang]{Department of Applied Mathematics, Anhui University of Technology, Ma'anshan, 243002, P.~R.~China}
\email{zhangminmin0907@163.com}
\thanks{*Corresponding author.}

\subjclass[2010]{Primary 28A80; 42C05}

\keywords{Self-affine measure; spectral measure; Beurling dimension; Beurling density}


\begin{abstract}
It is well-known that the Beurling dimension of the spectra of certain singular continuous spectral measures possesses an intermediate property.  In this paper, we establish that for a class of self-affine spectral measures $\mu$, both the Beurling dimension and Beurling density of their spectra attain full flexibility simultaneously. Specifically, for any $t\in (0,\dim_H^w(\supp(\mu)))$ and $s\in [0,\infty)$, there exists a spectrum $\Lambda:=\Lambda_{t,s}$ of $\mu$ satisfying
\[\dim_{Be}(\Lambda)=t\quad\text{and}\quad D_{t}^+(\Lambda)=s,\]
where $\dim_H^w$ denotes the pseudo Hausdorff dimension, $\dim_{Be}$ denotes the Beurling dimension and $D_{t}^+$ denotes the $t$-Beurling density. These results provide new insights into the structure of the spectra for a singular continuous spectral measure.
\end{abstract}

\maketitle

\section{Introduction}
We say that $\mu$ is a {\it spectral measure} on $\R^d$ if $\mu$ is a probability measure on $\R^d$ and there exists a discrete set $\Lambda\subseteq \R^d$ such that $E(\Lambda):=\{e^{-2\pi i\langle\lambda,x\rangle}\colon\lambda\in\Lambda\}$ forms an orthonormal basis for $L^2(\mu)$. The set $\Lambda$ is called a {\it spectrum} of $\mu$.

The theory of spectral measure originated from the well-known fact that the exponential functions $\{e^{-2\pi i\langle n,x\rangle}\colon n\in\Z^d\}$ form an orthonormal basis for $L^2([0,1]^d)$, i.e., the Lebesgue measure restricted to $[0,1]^d$ is a spectral measure, with $\Z^d$ serving as its spectrum and was initiated by the spectral set conjecture proposed by Fuglede \cite{Fug74} in 1974, during his study of extension problems for commuting partial differential operators. This conjecture asserts that the normalized Lebesgue measure restricted to a measurable set $\Omega$ with positive finite measure is a spectral measure if and only if $\Omega$ can tile the Euclidean space by translations. While the conjecture turned out to be false in general \cite{Tao04}, it triggered extensive exploration of spectral measures beyond the classical Lebesgue setting.

In 1998, Jorgensen and Pedersen \cite{JP98} constructed the first example of a singular, non-atomic spectral measure, i.e., the standard middle-fourth Cantor measure $\mu_{4,\{0,2\}}$ is a spectral measure. Since then, the spectrality of self-similar measures, self-affine measures and Moran measures has been extensively explored; see \cite{AH14,Dai12,DHLau14,DHLai19,DJ07a,DJ07b,LW02,LMW22} and references therein. Meanwhile, a variety of exotic phenomena completely absent in the Lebesgue setting were gradually discovered. First, the spectra of a singular continuous spectral measure are typically far from being lattices (they may display intricate tree-like structures \cite{ADH22,Dai16,DHLai13,DHS09}), and many of them have the cardinality of the continuum up to translations \cite{DHS09,FHW18}. Second, the spectra of many singular continuous spectral measures exhibit a scaling property: some integer multiples of a spectrum are also spectra \cite{ADH22,Dai16,DH16,HTW19,LW02,Str00} and the convergence of Fourier series strongly depends on the choice of the scaling spectra \cite{DHS14,Str00,Str06}. Third, for certain singular continuous spectral measures, the Beurling dimension of their spectra possesses an intermediate value property, attaining every value between zero and the maximal possible dimension \cite{DHSW11,LW1,Liu21,Lu25,TW21}.

To make the above third phenomenon precise, we recall the notion of Beurling dimension. For a discrete set $\Lambda\subseteq\R^d$ and $r>0$, the \textit{upper $r$-Beurling density} of $\Lambda$ is defined  by
 $$D_r^{+}(\Lambda)=\limsup\limits_{h\rightarrow\infty}\sup\limits_{x\in\R^d}\frac{\#(\Lambda\cap B(x, h))}{h^r},$$
where $\#E$ is the cardinality of the set $E$ and $B(x, h)$ is the open ball with center $x$ and radius $h$. The {\it Beurling dimension} of $\Lambda$ is then
\[
\dim_{Be}(\Lambda)=\inf\{r>0 : D_r^+(\Lambda)=0\}=\sup\{r>0 : D_r^+(\Lambda)=\infty\}.
\]
In the classical setting, the Beurling density  plays a key role (e.g., the classic density result of Landau \cite{Lan67}), whereas for singular continuous spectral measures the Beurling dimension has become a fundamental tool in the study of spectral structure. In particular, for a class of Moran spectral measures including $\mu_{4,\{0,2\}}$, an intermediate value property was established in \cite{LW1}. We present the precise result  for $\mu_{4,\{0,2\}}$.
\begin{theorem}[\cite{DHSW11,LW1}]\label{interddd}
For any orthogonal set $\Lambda$ of $\mu_{4,\{0,2\}}$, i.e., $E(\Lambda)$ is orthogonal in $L^2(\mu_{4,\{0,2\}})$, we have that $\dim_{Be}(\Lambda)\leq \frac{1}{2}$. Moreover, for any $t\in [0,\frac{1}{2}]$, there exists a spectrum $\Lambda_t$ such that $\dim_{Be}(\Lambda_t)=t$.
\end{theorem}

Based on this,  Dai \footnote{It is worth pointing out that we obtained Theorem \ref{interddd} in 2020 (published in 2024), and shortly after, Dai raised Question \ref{queb}.} \cite{Dai23} subsequently posed the following question:
\begin{question}\label{queb}
For any $(t,s)\in(0,\frac{1}{2})\times[0,\infty)$, does there exist a spectrum $\Lambda:=\Lambda_{t,s}$ of $\mu_{4,\{0,2\}}$ such that
\[\dim_{Be}(\Lambda)=t\quad \text{and}\quad D_{t}^+(\Lambda)=s?\]
\end{question}
\begin{remark}
When $t=0$, it is easy to see that $s$ can only be infinity. When $t=\frac{1}{2}$, Wu and the second author proved in \cite{WW24} that for any regular spectrum $\Lambda$ of $\mu_{4,\{0,2\}}$, $D_{\frac{1}{2}}^+(\Lambda)\leq \sqrt{6}$ and this bound is attainable; they also conjectured that the same bound holds for all spectra. If this conjecture is true, then for $t=\frac{1}{2}$ one would expect $s\in[0,\sqrt{6}]$.
\end{remark}
We conjecture that the answer to Question \ref{queb} is affirmative; however, it remains unresolved and persists as an open problem in the current literature. Instead, by substituting $\mu_{4,\{0,2\}}$ with a class of self-affine measures on the plane, we establish analogous conclusions to those in Question \ref{queb}.

Let $R\in M_n(\Z)$ be an expanding matrix, i.e., all eigenvalues of the integer matrix $R$ have moduli strictly larger than $1$. Let $D\subseteq \Z^d$ be a finite subset. Then by Hutchinson's theorem \cite{Fal14}, there exists a unique Borel probability  measure $\mu:=\mu_{R,B}$, which satisfies
\begin{equation}\label{sefaf}
\mu(E)=\sum_{b\in B}\frac{1}{\#B}\mu\circ\varphi_b^{-1}(E)
\end{equation}
for all Borel subsets $E$ of $\R^d$, where $\varphi_b(x)=R^{-1}(x+b)$. Such a measure $\mu_{R,B}$ is called a {\it self-affine measure}.

In this paper, we consider the self-affine measure $\mu_{R,B}$ generated by
\begin{equation*}
\text{the expanding matrix}~R\in M_2(\Z)~\text{and the digit set}~ B=\left\{\begin{pmatrix}
		0 \\
		0 \\
	\end{pmatrix},\begin{pmatrix}
		1 \\
		0 \\
	\end{pmatrix}\right\}.
 \end{equation*}
There are many papers investigating this model and a wide range of specific cases were considered; see \cite{Li07,Li12,Li16,LWen12,WA23} and references therein. While in general it is very difficult to completely determine the spectrality of this model, it is known that for the specific choice $R=\begin{pmatrix}
		2 & 1 \\
		0 & 2 \\
	\end{pmatrix}$, the measure $\mu_{R,B}$ is spectral and all of its spectra can be described explicitly in \cite{Li07}. For completeness, we present the proof in Section 2 (Theorem \ref{thsspc}). In this paper, we focus on this interesting case, i.e.,
\begin{equation}\label{seco}
R=\begin{pmatrix}
		2 & 1 \\
		0 & 2 \\
	\end{pmatrix}\quad \text{and}\quad B=\left\{\begin{pmatrix}
		0 \\
		0 \\
	\end{pmatrix},\begin{pmatrix}
		1 \\
		0 \\
	\end{pmatrix}\right\}
\end{equation}
 and tackle similar questions to Question \ref{queb}.

Combining Theorem 1.3 in \cite{TW21} and Theorem 1.3 in \cite{AL23}, we first obtain the following:
\begin{theorem}\label{thupp}
Let $\mu_{R,B}$ be the self-affine measure generated by the matrix and the digit set defined in \eqref{seco}. For any spectrum $\Lambda$ of $\mu_{R,B}$, $0\leq \dim_{Be}(\Lambda)\leq 1$, where the lower and upper bounds can be obtained.
\end{theorem}

On the basis of Theorem \ref{thupp}, we obtain an intermediate value property of spectra for $\mu_{R,B}$.
\begin{theorem}\label{MR1}
Let $\mu_{R,B}$ be the self-affine measure generated by the matrix and the digit set defined in \eqref{seco}. Then for any $t\in[0,1]$, there exists a spectrum $\Lambda_{t}$ of $\mu_{R,B}$ such that
\[\dim_{Be}(\Lambda_{t})=t.\]
\end{theorem}

Furthermore, we can determine the cardinality of the spectra whose Beurling dimensions are equal to a fixed value between $0$ and $1$. More precisely, we obtain the following result, which may be somewhat surprising.

\begin{theorem}\label{MR2}
Let $\mu_{R,B}$ be the self-affine measure generated by the matrix and the digit set defined in \eqref{seco}. For any $t\in [0,1]$, the level set
 \[
 L_t:=\{\Lambda\colon\text{$\Lambda$ is a spectrum of $\mu_{R,B}$ and $\dim_{Be}(\Lambda)=t$} \}
 \]
 has the cardinality of the continuum.
\end{theorem}

Finally, we arrive at the following result analogous to Question \ref{queb}.

\begin{theorem}\label{MR3}
Let $\mu_{R,B}$ be the self-affine measure generated by the matrix and the digit set defined in \eqref{seco}. Then for any $(t,s)\in(0,1)\times[0,\infty)$, there exists a spectrum $\Lambda_{t,s}$ of $\mu_{R,B}$ such that
\[\dim_{Be}(\Lambda_{t,s})=t~\text{and}~D_{t}^+(\Lambda_{t,s})=s.\]
\end{theorem}
\begin{remark}
When $t=0$, it is trivial that $s$ can only be infinity. When $t=1$, we have that $D_1^+(\Lambda)\leq 2$ for any spectrum $\Lambda$ of $\mu_{R,B}$, and the upper 1-Beurling density can attain every value in $[0,2]$ (see Proposition \ref{prop:t1}).

\end{remark}

The paper is organized as follows. In Section 2, we collect some known results and prove Theorem \ref{thupp}. In Section 3, we prove Theorems \ref{MR1} and \ref{MR2}. In Section 4, we prove Theorem \ref{MR3} and Proposition \ref{prop:t1}. In Section 5, we give some further questions.

\section{Preliminaries}

\subsection{Spectrality and spectra}
For a probability measure $\mu$ on $\R^2$, its Fourier transform is defined as usual,
      $$
      \hat{\mu}(\xi)=\int e^{-2\pi i\langle\xi,x\rangle}\mathrm{d}\mu(x).
      $$
 It is well-known that a discrete set $\Lambda$ is an orthogonal set of $\mu$ if and only if
      \begin{equation}\label{eqore}
      (\Lambda-\Lambda)\setminus\{0\}\subseteq\mathcal{Z}(\hat{\mu}),
      \end{equation}
 where $\mathcal{Z}(f):=\{x\in\R^2\colon f(x)=0\}$. Since the properties of orthogonal sets are invariant under translations, we always assume that $0\in\Lambda$, and hence $\Lambda\subseteq\Lambda-\Lambda$.

From \eqref{sefaf} we have
\[
\hat{\mu}_{R,B}(\xi)=\prod_{j=1}^{\infty}m_{B}((R^t)^{-j}\xi),~\xi\in\mathbb{R}^2,
\]
where $R^t$ denotes the transpose of $R$ and
\[
m_{B}(x)=\frac{1}{\#B}\sum_{b\in B}e^{-2\pi i\langle b,x\rangle}.
\]
It is easy to see that the zero set of $m_{B}$ is
\[
\mathcal{Z}(m_{B})=\left\{\begin{pmatrix}
		\xi_1\\
		\xi_2\\
	\end{pmatrix}\colon \xi_1\in \frac{2\Z+1}{2},\xi_2\in\R\right\}
\]
and therefore
\begin{equation}\label{eqmle330}
\begin{split}
\mathcal{Z}(\hat{\mu}_{R,B})&=\{\xi\colon \hat{\mu}_{R,B}(\xi)=0\}=\bigcup_{k=1}^{\infty}(R^t)^k\mathcal{Z}(m_{B})\\
&=\bigcup_{k=1}^{\infty}\left\{\begin{pmatrix}
		2^k&0\\
		k2^{k-1}&2^k\\
	\end{pmatrix}\begin{pmatrix}
		\xi_1\\
		\xi_2\\
	\end{pmatrix}\colon \xi_1\in \frac{2\Z+1}{2},\xi_2\in\R\right\}\\
&=\bigcup_{k=1}^{\infty}\left\{\begin{pmatrix}
		2^k\xi_1\\
		k2^{k-1}\xi_1+2^k\xi_2\\
	\end{pmatrix}\colon \xi_1\in \frac{2\Z+1}{2},\xi_2\in\R\right\}\\
&=\left\{\begin{pmatrix}
		\xi_1\\
		\xi_2\\
	\end{pmatrix}\colon \xi_1\in \Z\setminus\{0\},\xi_2\in\R\right\}.
\end{split}
\end{equation}

For any $A\subseteq \Z$ containing $0$, and any function $\beta\colon A\rightarrow\R$ with $\beta(0)=0$, we define
\begin{equation}\label{ortadd}
\Lambda_{A,\beta}=\left\{\begin{pmatrix}
		n\\
		\beta(n)\\
	\end{pmatrix}\colon n\in A,\beta(n)\in\R\right\}.
\end{equation}
Based on the zero set of $\hat{\mu}_{R,B}$, we have the following simple conclusion.

\begin{lemma}\label{leorr}
Let $\Lambda$ be a subset of $\R^2$ with $0\in \Lambda$. Then $\Lambda$ is an orthogonal set of $\mu_{R,B}$ if and only if there exist a set $A\subseteq \Z$ containing $0$ and a function $\beta\colon A\rightarrow\R$ with $\beta(0)=0$ such that $\Lambda=\Lambda_{A,\beta}$, where $\Lambda_{A,\beta}$ is defined as in \eqref{ortadd}.
\end{lemma}
\begin{proof}
It follows directly from \eqref{eqore} and the explicit form of $\mathcal{Z}(\hat{\mu}_{R,B})$ given in \eqref{eqmle330}.
\end{proof}

To determine when an orthogonal set is actually a spectrum, following the idea of Parseval's identity method, Jorgensen and Pedersen \cite{JP98} gave a criterion  by using the function
$$Q_\Lambda(\xi):=\sum_{\lambda\in\Lambda}|\hat{\mu}(\xi+\lambda)|^2.$$

\begin{theorem}[\cite{JP98}]\label{thsspcaa}
Let $\mu$ be a Borel probability measure with compact support in $\R^2$, and let $\Lambda\subseteq \R^2$ be a countable subset. Then $\Lambda$ is a spectrum of $\mu$ if and only if $Q_\Lambda(\xi)\equiv
    1$ for $a.e.~\xi\in\R^2$.
    \end{theorem}

    The following theorem involving the spectrality of $\mu_{R,B}$ is due to Li \cite{Li07}. For completeness, we give its proof here.
    \begin{theorem}\label{thsspc}
Let $\beta\colon \Z\rightarrow\R$ be an arbitrary function with $\beta(0)=0$, and let $\Lambda_{\Z,\beta}$ be defined as in \eqref{ortadd}. Then $\Lambda_{\Z,\beta}$ is a spectrum of $\mu_{R,B}$.
\end{theorem}
\begin{proof}
By an easy calculation, for $\xi=(\xi_1,\xi_2)^t\in \R^2$ where $a^t$ denotes the transpose of a vector $a\in\R^2$, $\hat{\mu}_{R,B}(\xi)=e^{-\pi i\xi_1}\prod_{j=1}^\infty\cos(\frac{\pi\xi_1}{2^j})=e^{-\pi i\xi_1}\frac{\sin(\pi\xi_1)}{\pi\xi_1}$ and then if $\xi_1\in \R\setminus \Z$,
\begin{equation*}
\begin{split}
Q_{\Lambda_{\Z,\beta}}(\xi)&=\sum_{\lambda\in\Lambda_{\Z,\beta}}|\hat{\mu}_{R,B}(\xi+\lambda)|^2\\
&=\sum_{\lambda_1\in\Z}\left|\frac{\sin(\pi(\lambda_1+\xi_1))}{\pi(\lambda_1+\xi_1)}\right|^2\\
&=1,
\end{split}
\end{equation*}
where the last equality holds by the known identity $\sum_{\lambda_1\in \mathbb{Z}} \frac{1}{(\lambda_1 + \xi_1)^2} = \frac{\pi^2}{\sin^2(\pi \xi_1)}$. By Theorem \ref{thsspcaa}, we have that $\Lambda_{\Z,\beta}$ is a spectrum of $\mu_{R,B}$.
\end{proof}
\begin{remark}
Combining Lemma \ref{leorr} and Theorem \ref{thsspc}, we have that all spectra of $\mu_{R,B}$ are known. Precisely, all of the spectra containing zero of $\mu_{R,B}$ are of the following forms:
\begin{equation}\label{spectrumf}
\Lambda_\beta:=\Lambda_{\Z,\beta}=\left\{\begin{pmatrix}
		\lambda_1 \\
		\beta(\lambda_1)
	\end{pmatrix}\in\R^2\colon \lambda_1\in\Z\right\}
\end{equation}
where $\beta:\Z\rightarrow \R$ is an arbitrary function with $\beta(0)=0$.
\end{remark}

\subsection{Beurling density and Beurling dimension}
As introduced in the Introduction, for a countable set $\Lambda\subseteq \R^2$ and $r>0,$ the \textit{upper $r$-Beurling density} of $\Lambda$ is defined by
\[
D_r^+(\Lambda)=\limsup\limits_{h\to \infty}\sup_{x\in \R^d}\frac{\#(\Lambda \cap B(x,h))}{h^r}.
\]
It is easy to prove that there is a critical value of $r$ where $D_r^+(\Lambda)$ jumps from $\infty$ to $0$. This critical value is defined as the \textit{Beurling dimension} of $\Lambda$, or more precisely,
\begin{eqnarray*}
\dim_{Be} (\Lambda)=\inf\{r:D_r^+(\Lambda)=0\}=\sup\{r:D_r^+(\Lambda)=\infty\}.
\end{eqnarray*}

Now, we prove Theorem \ref{thupp}. To prove it we need the following two theorems.

\begin{theorem}\label{th1euur}
Let $\mu_{R,B}$ be the self-affine measure generated by the matrix and the digit set defined in \eqref{seco}. Then for any orthogonal set $\Lambda$ of $\mu_{R,B}$, we have
$$\dim_{Be}(\Lambda)\leq\dim_{H}^{w}(\supp(\mu_{R,B}))=\frac{\ln\#B}{\ln \lambda_{\min}}=1,$$
where $\lambda_{\min}$ is the minimum moduli of the eigenvalues of $R$.
\end{theorem}
\begin{proof}
It is a direct corollary of Theorem 1.3 in \cite{TW21}.
\end{proof}

\begin{theorem}\label{th1euur1}
There exists a spectrum of $\mu_{R,B}$  with Beurling dimenison zero.
\end{theorem}
\begin{proof}
Let $L=\left\{\begin{pmatrix}
		0 \\
		0 \\
	\end{pmatrix},\begin{pmatrix}
		1 \\
		0 \\
	\end{pmatrix}\right\}$.
Then it is easy to see that \[
H = \frac{1}{\sqrt{q}} \begin{bmatrix} e^{-2\pi i \langle R^{-1} b, l \rangle} \end{bmatrix}_{b \in B, l \in L}
\]
is unitary. By Theorem 1.3 in \cite{AL23}, we obtain the desired result.
\end{proof}

By Theorems \ref{th1euur} and \ref{th1euur1}, it remains to prove that the upper bound can be attained; this part will be deferred to the proof of Theorem \ref{MR1}.

For later use we also introduce a ``ball density'' centered at the origin.  For a countable set $\Lambda\subseteq\R^2$ and $r>0$, the \textit{upper $r$-density} of $\Lambda$ is defined by
\[
B_r^+(\Lambda)=\limsup\limits_{h\to \infty}\frac{\#(\Lambda \cap B(0,h))}{h^r}.
\]
Similar to the Beurling dimension, we can define the {\it Banach dimension} of $\Lambda$ by
\begin{eqnarray*}
\dim_{Ba} (\Lambda)=\inf\{r:B_r^+(\Lambda)=0\}=\sup\{r:B_r^+(\Lambda)=\infty\}.
\end{eqnarray*}

We end this section with the following lemma which will be needed in the following sections. \begin{lemma}\label{lembab}
Let $\beta:\Z\to\R$ be an odd function, strictly increasing on $\Z$.
Assume that one of the following two conditions holds:
\begin{enumerate}[(i)]
\item  there exist constants $p>1$, $c>0$ and an integer $N_0\in\N$ such that
$\beta(n)=c\,n^{p}$ for all $n\ge N_0$;
\item  there exist $a>1$ and $N_0\in\N$ such that $\beta(n)=a^{n}$ for all $n\ge N_0$.
\end{enumerate}
Then $\dim_{Be}(\Lambda_{\beta})=\dim_{Ba}(\Lambda_{\beta})$, where $\Lambda_{\beta}$ is defined as in \eqref{spectrumf}.
\end{lemma}

\begin{proof}
We give the detailed proof for case (i); the proof of case (ii) is analogous, with $h^{1/p}$ replaced by $\log_{a}h$ and we omit the proof.

Without loss of generality, we assume $\beta(n)=c\,n^{p}$ for all $n\ge 0$ and, by oddness, $\beta(n)=-c\,|n|^{p}$ for $n<0$.
Let $\Lambda=\Lambda_\beta$ and fix an $x=(x_{1},x_{2})\in\R^{2}$ and a radius $h>0$.  Write
\[
S_{+}=\{n\ge 0: (n,\beta(n))\in B(x,h)\},\qquad
S_{-}=\{n<0: (n,\beta(n))\in B(x,h)\}.
\]
For $n\ge 0$ the condition $(n-x_{1})^{2}+(c n^{p}-x_{2})^{2}<h^{2}$ implies in particular
$|c n^{p}-x_{2}|<h$.  Thus
\[
S_{+}\subseteq I:=\{n\ge 0: c n^{p}\in (x_{2}-h,\,x_{2}+h)\},
\]
and because $\beta$ is strictly increasing, $I$ is a set of consecutive integers (an interval in $\Z$).
Let $n_{1}\le n_{2}$ be the smallest and largest elements of $I$ (if $I=\varnothing$ then $\#S_{+}=0$).
From $c n_{2}^{p}<x_{2}+h$ and $c n_{1}^{p}>x_{2}-h$ we obtain
\begin{equation*}
n_{1}\geq \Bigl(\frac{(x_{2}-h)_{+}}{c}\Bigr)^{1/p},\qquad
n_{2}<\Bigl(\frac{(x_{2}+h)_{+}}{c}\Bigr)^{1/p},
\end{equation*}
where $t_{+}=\max\{t,0\}$.
Consequently
\[
\#S_{+}\le n_{2}-n_{1}+1\le
\Bigl(\frac{(x_{2}+h)_{+}}{c}\Bigr)^{1/p}-\Bigl(\frac{(x_{2}-h)_{+}}{c}\Bigr)^{1/p}+1.
\]

We now estimate the difference of the two power functions.
If $x_{2}\le 2h$, then
\[
\Bigl(\frac{(x_{2}+h)_+}{c}\Bigr)^{1/p}\le \Bigl(\frac{3h}{c}\Bigr)^{1/p},
\]
and thus
\[
\#S_{+}\le \Bigl(\frac{3}{c}\Bigr)^{1/p} h^{1/p}+1.
\]
If $x_{2}>2h$, then $x_{2}-h>h>0$ and the mean value theorem applied to $t\mapsto t^{1/p}$ on
$[x_{2}-h,\,x_{2}+h]$ yields
\[
(x_{2}+h)^{1/p}-(x_{2}-h)^{1/p}
=\frac{1}{p}\,\xi^{\frac{1}{p}-1}\,(2h)
\]
for some $\xi\in[x_{2}-h,\,x_{2}+h]$.  Because $\xi\ge x_{2}-h>h$, we have
$\xi^{\frac{1}{p}-1}\le h^{\frac{1}{p}-1}$.  Hence
\[
\Bigl(\frac{x_{2}+h}{c}\Bigr)^{1/p}-\Bigl(\frac{x_{2}-h}{c}\Bigr)^{1/p}
\le \frac{2}{p\,c^{1/p}}\,h^{1/p}.
\]
Therefore in every case there exists a constant $C_{1}=C_{1}(c,p)>0$ such that
\[
\#S_{+}\le C_{1} h^{1/p}+1,
\]
for all sufficiently large $h$. By symmetry the same estimate holds for $S_{-}$, and we obtain
\begin{equation*}
\#\bigl(\Lambda\cap B(x,h)\bigr)=\#S_{+}+\#S_{-}\le 2C_{1} h^{1/p}+3\le (2C_1+3)h^{1/p}=: C_{2} h^{1/p}
\end{equation*}
for large $h$, uniformly in $x\in\R^{2}$.

On the other hand, we give a lower bound for the ball centered at the origin.
For $n\ge 0$ we have $(n,c n^{p})\in B(0,h)$ as soon as both $n<h$ and $c n^{p}<h$.
Choosing $n\le \min\{h/2,\,(h/(2c))^{1/p}\}$ guarantees this.  Since $p>1$, for large $h$
the condition $n\le (h/(2c))^{1/p}$ is the more restrictive one.  Hence
\[
\#\bigl(\Lambda\cap B(0,h)\bigr)\ge 2\Bigl\lfloor\Bigl(\frac{h}{2c}\Bigr)^{1/p}\Bigr\rfloor+1
\ge C_{3} h^{1/p}
\]
with $C_{3}>0$ for all large $h$. Here and in the sequel, $\lfloor x\rfloor$ denotes the largest integer that is no greater than $x$.

Combining the upper and lower estimates we find a constant $K>0$ such that
\[
\#\bigl(\Lambda\cap B(x,h)\bigr)\le K\;\#\bigl(\Lambda\cap B(0,h)\bigr)
\]
for all $x\in\R^{2},\;h\ge h_{0}$. Since $B_{r}^{+}(\Lambda)\le D_{r}^{+}(\Lambda)$ trivially, it follows that $\dim_{Be}(\Lambda)=\dim_{Ba}(\Lambda)$.
\end{proof}

\section{Proofs of Theorems \ref{MR1} and \ref{MR2}}

\begin{proof}[Proof of Theorem \ref{MR1}]
When $t=0$, let $\beta(n)=\text{sgn}(n)\cdot 2^{|n|},n\in \Z$, where \text{sgn} is the symbolic function. Take $r,h>0$. Then there exists $n\in \Z$ such that $2^n<h\leq 2^{n+1}$. Consequently,
\begin{eqnarray*}
\frac{\#(\Lambda_{\beta}\cap B(0,h))}{h^r}\leq \frac{\#(\Lambda_{\beta}\cap B(0,2^{n+1}))}{2^{nr}}\leq\frac{2(n+1)}{2^{nr}}\rightarrow0~\text{as}~h\to \infty,
\end{eqnarray*}
which implies that
\begin{eqnarray*}
B_r^+(\Lambda_{\beta})=\limsup\limits_{h\to \infty}\frac{\#(\Lambda_{\beta} \cap B(0,h))}{h^r}=0,
\end{eqnarray*}
and thus $\dim_{Ba}(\Lambda_{\beta})=0$. By Lemma \ref{lembab},  $\dim_{Be}(\Lambda_{\beta})=\dim_{Ba}(\Lambda_{\beta})=0$. This is an explicit constructive proof of Theorem \ref{th1euur1}.

When $t=1$, let $\beta(n)=n,n\in \Z$. Now we prove $\dim_{Be}(\Lambda_{\beta})=1$. By Theorem \ref{thupp} we only need to prove that $\dim_{Be}(\Lambda_{\beta})\geq1$. Take $h>0$, choose $n\in \Z$ such that $n<h\leq n+1$. Then for any $x\in \R^2$,
\begin{eqnarray*}
\frac{\#(\Lambda_{\beta}\cap B(x,h))}{h}\geq \frac{\#(\Lambda_{\beta}\cap B(x,n))}{n+1}.
\end{eqnarray*}
Then
\begin{eqnarray*}
\sup_{x\in \R^2}\frac{\#(\Lambda_{\beta}\cap B(x,h))}{h}\geq \frac{\#(\Lambda_{\beta}\cap B(0,n))}{n+1}=\frac{2\Big\lfloor \frac{n}{\sqrt{2}}\Big\rfloor+1}{n+1}\geq \frac{\sqrt{2}n-1}{n+1},
\end{eqnarray*}
which implies that $D_1^+(\Lambda_{\beta})\geq \sqrt{2}>0$. Therefore, $\dim_{Be}(\Lambda_{\beta})\geq1$. Combined with the upper bound
$\dim_{Be}(\Lambda_{\beta})\leq1$ from Theorem \ref{thupp}, we obtain
$\dim_{Be}(\Lambda_{\beta})=1$; in particular, the upper bound in
Theorem \ref{thupp} is attainable.

We are left to prove the case that $t\in (0,1)$. Define $\beta(n)=\mathrm{sgn}(n)|n|^{1/t}$. We first estimate $\#(\Lambda_{\beta}\cap B(0,h))$. For sufficiently large $h$, if $|n|^{1/t}\leq h/2$, then
\[
n^2 + |n|^{2/t} \leq (h/2)^{2t} + (h/2)^2 \leq h^2/2 < h^2,
\]
so all such points lie in $B(0,h)$.

On the other hand, if the point lies in $B(0,h)$, then $|n|^{1/t}\leq h$. Hence there exist constants $c_1,c_2>0$ such that for all large $h$,
\[
c_1 h^t \leq \#(\Lambda_{\beta}\cap B(0,h)) \leq c_2 h^t.
\]
In particular, $B_t^+(\Lambda_{\beta})$ is finite and positive, which gives $\dim_{Ba}(\Lambda_{\beta})=t$. By Lemma \ref{lembab}, $\dim_{Be}(\Lambda_{\beta})=\dim_{Ba}(\Lambda_{\beta})=t$.

\end{proof}

\begin{proof}[Proof of Theorem \ref{MR2}]
When $t=0$, for any real number $a>1$, let $\beta(n)=\text{sgn}(n)\cdot a^{|n|},n\in \Z$. Then we have $\dim_{Be}(\Lambda_{\beta})=0$. In fact, the proof is identical to that of the first case $a=2$ of Theorem \ref{MR1}. As $a$ varies through $(1,+\infty)$, the set $L_t$ is shown to possess uncountable cardinality.

When $0<t\leq 1$, for any real number $a>1$, let $\beta(n)=\text{sgn}(n)|an|^{\frac{1}{t}}$. Likewise, as in the proof of Theorem \ref{MR1},  we have $\dim_{Be}(\Lambda_{\beta})=t$. As $a$ varies through $(1,+\infty)$, we obtain the desired result.

\end{proof}

\section{Proofs of Theorem \ref{MR3} and Proposition \ref{prop:t1}}
\begin{proof}[Proof of Theorem \ref{MR3}]
Let $0<t<1$ and $s\in[0,\infty)$.  We construct a spectrum
$\Lambda=\Lambda_\beta$ of the form \eqref{spectrumf} by choosing a
suitable odd function $\beta:\mathbb{Z}\to\mathbb{R}$ with $\beta(0)=0$.

\textbf{Case 1: $s=0$.}
Define
\[
\beta(n)=\operatorname{sgn}(n)\,|n|^{1/t}\log|n|\ \ (n\neq0),\qquad \beta(0)=0 .
\]
Let $\psi:[0,\infty)\to[0,\infty)$ be the inverse of the increasing
function $x\mapsto x^{1/t}\log x$ on $[1,\infty)$, extended by
$\psi(y)=0$ for $y<\log 1$ if needed.  Standard asymptotic inversion
gives
\begin{equation*}
\psi(y)\sim\Bigl(\frac{y}{\log y}\Bigr)^{t},\qquad
\psi'(y)\sim ty^{t-1}(\log y)^{-t}\quad (y\to\infty).
\end{equation*}
Hence we can pick constants $C_{1},C_{2},Y_{0}>0$ such that for all
$y\ge Y_{0}$,
\[
\psi(y)\le C_{1}\Bigl(\frac{y}{\log y}\Bigr)^{t},\qquad
\psi'(y)\le C_{2}\,y^{t-1}(\log y)^{-t}.
\]

For the ball centered at the origin we have
$\sqrt{n^{2}+\beta(n)^{2}}\sim|n|^{1/t}\log|n|$.  Consequently there exist
$c_{1},c_{2}>0$ such that for all large $h$,
\[
c_{1}\Bigl(\frac{h}{\log h}\Bigr)^{t}
\le \#\bigl(\Lambda\cap B(0,h)\bigr)
\le c_{2}\Bigl(\frac{h}{\log h}\Bigr)^{t}.
\]
Thus $B_{t}^{+}(\Lambda)=0$ and $B_{r}^{+}(\Lambda)=\infty$ for every $r<t$.

Now fix $x=(x_{1},x_{2})\in\mathbb{R}^{2}$ and $h>0$.  Points of $\Lambda$
in $B(x,h)$ satisfy $|\beta(n)-x_{2}|<h$.  The number $N_{+}(x_{2},h)$ of
non-negative such $n$ is bounded by
\[
N_{+}(x_{2},h)\le \psi\bigl(x_{2}+h\bigr)-\psi\bigl((x_{2}-h)_{+}\bigr)+1 .
\]
Set $a=\max\{x_{2},0\}$. Then $N_{+}\le \psi(a+h)-\psi((a-h)_{+})+1$.
For $h$ large enough we may assume all arguments of $\psi$ appearing below
are at least $Y_{0}$.

If $a\ge 2h$, then $(a-h)_{+}=a-h\ge h$.  By the mean value theorem and
the estimate on $\psi'$,
\begin{align*}
\psi(a+h)-\psi(a-h) &\le 2h\,\psi'(a-h) \\
&\le 2C_{2}\,h\,(a-h)^{t-1}(\log(a-h))^{-t} \\
&\le 2C_{2}\,h^{t}(\log h)^{-t}.
\end{align*}

If $0\le a<2h$, then
\[
\psi(a+h)-\psi((a-h)_{+}) \le \psi(3h)
\le C_{1}\Bigl(\frac{3h}{\log(3h)}\Bigr)^{t}
\le C_{3}\,h^{t}(\log h)^{-t}
\]
for some $C_3$.

Hence in all cases $N_{+}\le C_{4}(h/\log h)^{t}+1$.  By symmetry the same
holds for the non-positive integers, and since restricting also
$|n-x_{1}|<h$ only decreases the count, we obtain the uniform bound
\[
\sup_{x\in\mathbb{R}^{2}}\#\bigl(\Lambda\cap B(x,h)\bigr)
\le C_{5}\Bigl(\frac{h}{\log h}\Bigr)^{t}\qquad(h\ge H_{0}).
\]
Therefore $D_{t}^{+}(\Lambda)=0$, and for $r>t$ also $D_{r}^{+}(\Lambda)=0$.
Combined with the lower bound $B_{r}^{+}(\Lambda)=\infty$ for $r<t$, we
get $\dim_{\mathrm{Be}}(\Lambda)=t$.

\medskip
\textbf{Case 2: $0<s<\infty$.}
Define
\[
\beta(n)=\operatorname{sgn}(n)\Bigl(\frac{2|n|}{s}\Bigr)^{1/t}.
\]
Set $\phi(u)=(2u/s)^{1/t}$ for $u\ge0$; $\phi$ is strictly increasing,
$\phi^{-1}(y)=\frac{s}{2}y^{t}$, and $\phi^{-1}$ is concave because $0<t<1$.

Fix $x=(x_{1},x_{2})$ and $h>0$.  Exactly as above,
$N(x_{2},h):=\#\{n:|\beta(n)-x_{2}|<h\}$ satisfies
\[
N(x_{2},h)\le \frac{s}{2}F(x_{2},h)+2,
\]
where
\[
F(x_{2},h)=(x_{2}+h)^{t}-((x_{2}-h)_{+})^{t}
          +((-x_{2}+h)_{+})^{t}-((-x_{2}-h)_{+})^{t}.
\]
A short case study shows $F(x_{2},h)\le 2h^{t}$ (use sub-additivity of
$u\mapsto u^{t}$ when $|x_{2}|\ge h$, and calculus when $|x_{2}|<h$).
Thus $\sup_{x}\#(\Lambda\cap B(x,h))\le s h^{t}+2$, giving
$D_{t}^{+}(\Lambda)\le s$.

For the reverse inequality take $x=0$.  Given $\varepsilon>0$, choose
$\delta>0$ with $(1+\delta)^{-t/2}>1-\varepsilon$.  If
$|n|\le \frac{s}{2}h^{t}$, then $|\beta(n)|\le h$ and, for large $h$,
$n^{2}\le\delta h^{2}$, so the point lies in
$B(0,\sqrt{1+\delta}\,h)$.  Hence
\[
\#\bigl(\Lambda\cap B(0,\sqrt{1+\delta}\,h)\bigr)
\ge 2\Bigl\lfloor\frac{s}{2}h^{t}\Bigr\rfloor+1
= s(1+\delta)^{-t/2}(\sqrt{1+\delta}\,h)^{t}+O(1).
\]
Dividing by $(\sqrt{1+\delta}\,h)^{t}$, taking the supremum over $x$ and
letting $h\to\infty$ yields $D_{t}^{+}(\Lambda)\ge s(1+\delta)^{-t/2}
\ge s(1-\varepsilon)$.  Since $\varepsilon$ is arbitrary,
$D_{t}^{+}(\Lambda)\ge s$, whence $D_{t}^{+}(\Lambda)=s$ and
$\dim_{Be}(\Lambda)=t$.
\end{proof}

\begin{prop}\label{prop:t1}
Let $\mu_{R,B}$ be the self-affine measure generated by the matrix and the digit set defined in \eqref{seco}. For any spectrum $\Lambda$ of $\mu_{R,B}$, we have $D_1^+(\Lambda)\le 2$, and this upper bound is attained. Moreover, for every $s\in[0,2]$ there exists a spectrum $\Lambda$ with $\dim_{Be}(\Lambda)=1$ and $D_1^+(\Lambda)=s$.
\end{prop}
\begin{proof}
Let $\Lambda=\Lambda_\beta$ be a spectrum of $\mu_{R,B}$. Since the first coordinates of points in $\Lambda$ are integers, for any ball $B(x,h)$ with $x=(x_1,x_2)$ we have
\[
\#(\Lambda\cap B(x,h))\le \#\big(\mathbb{Z}\cap(x_1-h,x_1+h)\big)\le 2h+1,
\]
and therefore $D_1^+(\Lambda)\le 2$. The value $2$ is attained by taking $\beta(n)\equiv0$, in which case $\Lambda=\{(n,0):n\in\mathbb{Z}\}$ and $\#(\Lambda\cap B(0,h))=2\lfloor h\rfloor+1$, giving $D_1^+(\Lambda)=2$ and $\dim_{Be}(\Lambda)=1$. 

It remains to show that for every $s\in[0,2)$ there exists a spectrum $\Lambda$ with $\dim_{Be}(\Lambda)=1$ and $D_1^+(\Lambda)=s$. We consider  two cases.

\textbf{Case $s=0$.}
Define $\beta(0)=0$ and $\beta(n)=n\log|n|$ for $n\neq 0$.
Then $\beta$ is odd and strictly increasing on $\mathbb{Z}$.
Let $\phi(u)=u\log u$ for $u\ge 1$; $\phi$ is strictly increasing and maps $[1,\infty)$ onto $[0,\infty)$.
Denote by $\psi:[0,\infty)\to[1,\infty)$ its inverse, with the convention $\psi(0)=1$.
For large $y$ one has $\psi(y)=y/\log y+O(y/(\log y)^2)$ and
$\psi'(y)=1/(\log\psi(y)+1)\sim 1/\log y$.
In particular, there exist constants $Y_0>0$ and $C_0>0$ such that
\begin{equation}\label{eq:psi-est}
\psi(y)\le \frac{2y}{\log y},\qquad \psi'(y)\le \frac{2}{\log y}
\qquad\text{for all } y\ge Y_0 .
\end{equation}

Fix $x=(x_1,x_2)\in\mathbb{R}^2$ and $h>0$ large.
Every $\lambda=(n,\beta(n))\in\Lambda\cap B(x,h)$ must satisfy $|\beta(n)-x_2|<h$.
Let $N_+$ be the number of non-negative integers $n$ with this property.
If $x_2+h\le 0$ then $N_+=0$ and the estimates below hold trivially.
Set $a=\max\{x_2,0\}$.
For $n\ge 1$ we have $\beta(n)=\phi(n)$; thus the condition $|\beta(n)-x_2|<h$ is equivalent to
$\phi(n)\in (x_2-h,x_2+h)$.  Because $\phi$ is strictly increasing,
\begin{equation}\label{xha}
N_+\le \psi\bigl(a+h\bigr)-\psi\bigl((a-h)_+\bigr)+1.
\end{equation}

Now we estimate the right-hand side using \eqref{eq:psi-est}. If $a\ge 2h$, then $(a-h)_+\ge h\ge Y_0$ for large $h$.
By the mean value theorem,
\[
\psi(a+h)-\psi(a-h)\le 2h\sup_{\xi\in[a-h,a+h]}\psi'(\xi)
\le 2h\cdot\frac{2}{\log(a-h)}\le \frac{4h}{\log h}.
\]
If $0\le a<2h$, then
\[
\psi(a+h)-\psi((a-h)_+)\le \psi(3h)
\le \frac{2\cdot 3h}{\log(3h)}\le \frac{C h}{\log h}
\]
for some absolute constant $C$.

In both cases we obtain $N_+\le C_1 h/\log h + 1$ for large $h$, uniformly in $x_2$.
By symmetry the same bound holds for the non-positive integers, and adding them gives
\[
\sup_{x\in\mathbb{R}^2}\#\bigl(\Lambda\cap B(x,h)\bigr)
\le \frac{C_2\,h}{\log h}\qquad (h\ge H_0).
\]
Consequently $D_1^+(\Lambda)=0$, and also $D_r^+(\Lambda)=0$ for every $r\ge 1$.

It remains to show that $\dim_{Be}(\Lambda)=1$.
For the ball centered at the origin we use the asymptotic
$\sqrt{n^2+(n\log n)^2}=n\log n\bigl(1+o(1)\bigr)$.
Hence there exist constants $c_1,c_2>0$ such that, for all sufficiently large $h$,
\[
\bigl\{n\ge 1: n\log n\le c_1 h\bigr\}\subseteq B(0,h),
\]
and therefore
\[
\#\bigl(\Lambda\cap B(0,h)\bigr)
\ge 2\,\bigl\lfloor\psi(c_1 h)\bigr\rfloor+1
\ge c_2\,\frac{h}{\log h}.
\]
For any $r<1$ we obtain
\[
D_r^+(\Lambda)\ge \limsup_{h\to\infty}\frac{c_2\,h/\log h}{h^r}=\infty,
\]
which yields $\dim_{Be}(\Lambda)\ge 1$.  By Theorem \ref{thupp}, we have  $\dim_{Be}(\Lambda)=1$, which completes the proof for $s=0$.

\textbf{Case $0<s<2$.}
Set $c=\sqrt{4/s^2-1}>0$ and define $\beta(n)=c\,n$. Then $\Lambda=\{(n,cn):n\in\mathbb{Z}\}$. For this $\Lambda$, the distance from $(n,cn)$ to the origin is $|n|\sqrt{1+c^2}$, whence
\[
\#(\Lambda\cap B(0,h))=2\Big\lfloor\frac{h}{\sqrt{1+c^2}}\Big\rfloor+1,
\]
giving $B_1^+(\Lambda)=\frac{2}{\sqrt{1+c^2}}=s$.
For an arbitrary center $x=(x_1,x_2)$, the condition $(n-x_1)^2+(cn-x_2)^2<h^2$ implies that $n$ lies in an interval of length at most $2h/\sqrt{1+c^2}$; therefore
\[
\#(\Lambda\cap B(x,h))\le \Big\lfloor\frac{2h}{\sqrt{1+c^2}}\Big\rfloor+1.
\]
Taking the supremum over $x$ and the $\limsup$ as $h\to\infty$ gives $D_1^+(\Lambda)\le \frac{2}{\sqrt{1+c^2}}=s$. Together with the lower bound from $x=0$, we obtain $D_1^+(\Lambda)=s$.
Since $0<D_1^+(\Lambda)<\infty$, it follows that $\dim_{Be}(\Lambda)=1$.

Hence, we  complete the proof.
\end{proof}

\section{Further questions}
In this paper, we only consider $R=\begin{pmatrix}
		2 & 1 \\
		0 & 2 \\
	\end{pmatrix}$ for the digit set $B=\left\{\begin{pmatrix}
		0 \\
		0 \\
	\end{pmatrix},\begin{pmatrix}
		1 \\
		0 \\
	\end{pmatrix}\right\}$. There are many interesting problems left in the other general cases.

For example, $R=\begin{pmatrix}
		a & c \\
		0 & b \\
	\end{pmatrix}$
with $a\in 2\Z\setminus\{0\}$, $b\in 2\Z+1\setminus\{\pm1\}$ and $c\in \Z\setminus\{0\}$, $B$ is the same as in this paper. In this case, it is known that the corresponding self-affine measure $\mu_{R,B}$ is spectral \cite{Li16}. In general, the spectral structure is more complicated and is hard to show all of spectra explicitly like this model in this paper. A natural question is the following

\textbf{Question} Do all the results established in this paper retain their validity for the above model?

\subsection*{Acknowledgements}
The authors are grateful to the referee for his/her valuable comments and suggestions that led to the improvement of the manuscript. The project was supported by the National Natural Science Foundations of China (12271534,12301105).

\end{document}